\documentclass[12pt, a4paper]{amsart}
\usepackage{amsmath,amssymb,amsfonts, float}
\usepackage{thmtools}

\usepackage[all]{xy}
\usepackage{caption}
\usepackage{enumerate}
\usepackage{mathpazo}
\usepackage{a4wide}
\usepackage{diagbox}

\usepackage{bm}
\usepackage{graphicx}
\usepackage[breaklinks=true]{hyperref}

\usepackage{xcolor}
\usepackage{multicol}
\usepackage{tabularx}

\usepackage{hyperref}
\usepackage[capitalize]{cleveref}

\usepackage[normalem]{ulem}

\newtheorem{thm}{Theorem}[section] 
\newtheorem*{thm*}{Theorem} 
\newtheorem{prop}[thm]{Proposition}

\newtheorem{cor}[thm]{Corollary}

\theoremstyle{definition}
\newtheorem{definition}[thm]{Definition}

\newtheorem{question}[thm]{Question}
\newtheorem{rem}[thm]{Remark}

\DeclareMathOperator{\C}{\mathbb{C}}
\DeclareMathOperator{\Z}{\mathbb{Z}}

\DeclareMathOperator{\F}{\mathbb{F}}

\DeclareMathOperator{\Hom}{{\rm Hom}}

\DeclareMathOperator{\GL}{\text{GL}}

\DeclareMathOperator{\rank}{\text{rank}}

    \DeclareFontFamily{U}{wncy}{}
    \DeclareFontShape{U}{wncy}{m}{n}{<->wncyr10}{}
    \DeclareSymbolFont{mcy}{U}{wncy}{m}{n}
    \DeclareMathSymbol{\Sha}{\mathord}{mcy}{"58}

\numberwithin{equation}{section}

\DeclareSymbolFont{bbold}{U}{bbold}{m}{n}
\DeclareSymbolFontAlphabet{\mathbbold}{bbold}

\usepackage{hyperref}

\newcommand{\Rad}{{\rm Rad}}
\newcommand{\s}{\text{ss}}

\begin{document}
\title{On $U$-unitary Cayley graphs \\ over finite rings}
 \author{Tung T. Nguyen, Nguy$\tilde{\text{\^{e}}}$n Duy T\^{a}n }

 \address{Department of Mathematics, Elmhurst University, Illinois, USA}
 \email{tung.nguyen@elmhurst.edu}
 
  \address{
Faculty of Mathematics and Informatics, Hanoi University of Science and Technology, 1 Dai Co Viet Road, Hanoi, Vietnam } 
\email{tan.nguyenduy@hust.edu.vn}

\thanks{TTN is partially supported by an AMS-Simons Travel Grant. Parts of this work were completed during his trip to the JMM 2026 in Washington, DC. He thanks Elmhurst University for their support to attend this meeting through a travel grant.  NDT is partially supported by the Vietnam National
Foundation for Science and Technology Development (NAFOSTED) under grant number 101.04-2023.21}
\begin{abstract}
Graphs defined over a finite ring are well-studied in the literature. Due to their nature, these types of graphs connect several branches of mathematics, including algebra, number theory, matrix theory, and representation theory. In recent work, we studied $U$-unitary Cayley graphs over a finite commutative ring, which unifies several constructions of graphs with arithmetic origins. Among various structural graph-theoretic results on these graphs—such as their connectedness, primeness, and bipartiteness—we show that their spectra can be described via a certain supercharacter theory. Utilizing this spectral description, we are able to find some classes of gcd-graphs that possess perfect state transfer. In this article, we generalize this study to finite non-commutative rings, with a special focus on the case of the matrix rings with coefficients in a finite field. We show, in particular, that gcd-graphs over these matrix rings have no perfect state transfer.
\end{abstract}
\maketitle

\section{Introduction}

Let $n$ be a positive integer and $D$ a subset of proper divisors of $n$. The gcd-graph $G_n(D)$ is defined as follows: 

\begin{enumerate} 
\item The vertices of $G_n(D)$ are the elements of the finite ring $\Z/n$. 

\item Two vertices $a,b$ are adjacent if $\gcd(a-b, n) \in D.$  

\end{enumerate}
This type of graph was first introduced by Klotz and Sander in \cite{klotz2007some}. There, the authors describe several fundamental graph-theoretic properties of these graphs. In particular, they explain a beautiful connection between the spectra of these graphs and Ramanujan sums. As a consequence of this spectral description, Klotz and Sander show that all eigenvalues of gcd-graphs are integers. In \cite{so2006integral}, So proves the converse of this statement; namely, if a Cayley graph over $\mathbb{Z}/n$ has all integral eigenvalues, then it must be a gcd-graph. The works of Klotz, Sander, and So have led to a series of studies on Perfect State Transfer (PST) on graphs—a concept introduced by physicists in quantum spin networks (see \cite{bavsic2009perfect, BP2, saxena2007parameters, pst-gcd-new}).

In general, a gcd-graph can be defined over a finite commutative ring $R$ using the interplay between the additive and multiplicative structures of $R$ (see \cite{pst-gcd-new-2, minavc2024gcd, nguyengcd2026, pst-gcd-new}). We also refer to \cite{anderson2021graphs} and the references therein for a state-of-the-art overview of the study of graphs associated with rings. For gcd-graphs, in particular, many aspects of them have been investigated, including but not limited to their connectedness, primeness, and clique and independence numbers (see \cite{unitary, minavc2024gcd, nguyengcd2026, pst-gcd-new}). Furthermore, when the underlying ring $R$ is a finite commutative Frobenius ring, we show in previous work that the spectra of gcd-graphs over $R$ can be described by various arithmetical sums such as Gauss sums, Ramanujan sums, and Heilbronn sums (see \cite{nguyen2025supercharacters, nguyengcd2026}).

The goal of this article is to define and study the concept of gcd-graphs over arbitrary finite rings (and, more generally, $U$-unitary Cayley graphs over these rings). While some former results over commutative rings generalize straightforwardly, others require more careful consideration, as left and right multiplication might not be the same, which leads to several new phenomena, such as two-sided ideal structures. In particular, this necessitates a revisit of the notion of a gcd-graph in the non-commutative setting.

\subsection{Outline.} 
In \cref{sec:foundation}, we introduce the notion of a $U$-unitary Cayley graph over a finite ring $R$. We then describe some foundational graph-theoretic properties of these graphs, including their connectedness and primeness. In particular, we provide a complete answer regarding the connectedness and primeness of the unitary Cayley graph associated with $R$ under some mild conditions. \cref{sec:spectra} focuses on the spectra of these $U$-unitary Cayley graphs when the underlying $R$ is a symmetric Frobenius ring. We show that these spectra can be described via a certain supercharacter theory on $R$. As a consequence of this spectral description, we provide an upper bound for the number of distinct eigenvalues in a $U$-unitary Cayley graph. For certain graphs over a matrix ring $M_n(F)$, we also calculate this upper bound explicitly. In \cref{sec:spectra}, we also study the notion of relative Frobenius rings, which might be of independent interest. Finally, we utilize this spectral description to study the existence of Perfect State Transfer (PST) on $U$-unitary Cayley graphs. We show, in particular, that PST cannot exist on any gcd-graph over $M_n(F)$.

\section{$U$-unitary Cayley graphs and their graph theoretic properties} 
\label{sec:foundation}
\subsection{$U$-unitary Cayley graphs}
Let $R$ be a finite unital ring and $S$ a subset of $(R,+)$ such that $S = -S$ and $S$ does not contain $0$. The Cayley graph $\Gamma(R,S)$ is the undirected simple graph whose vertex set is $R$, and two vertices $a$ and $b$ are adjacent if and only if $a - b \in S$. In practice, $S$ is often referred to as the generating set of $\Gamma(R,S)$. In many applications, $S$ often has an arithmetic origin (see \cite{unitary, huang2022quadratic, jones2020isomorphisms, keshavarzi2025involutory, minac2023paley, podesta2019spectral, podesta2021_finitefield, suntornpoch2016cayley} for some works in this line of research).

When $R$ is commutative, we define in \cite{nguyen2025supercharacters} the notion of a $U$-unitary Cayley graph, where $U$ is a subgroup of $R^{\times}$—the unit group of $R$. More precisely, a Cayley graph of the form $\Gamma(R,S)$ is called a $U$-unitary Cayley graph if $S$ is stable under the action of $U$; namely, $US = S$. As we explain in \cite{nguyen2025supercharacters}, when $U = R^{\times}$ and $R = \mathbb{Z}/n$, this definition coincides with the classical definition of a gcd-graph described in the introduction. We remark that, by its definition, $U$ acts as an automorphism of each $U$-unitary Cayley graph $\Gamma(R,S)$ (see \cite[Proposition 3.7]{nguyen2025supercharacters}).

When $R$ is not commutative, we must take into account the fact that left and right multiplication might not be the same. Consequently, it seems reasonable to consider them simultaneously. For this reason, we introduce the following definition.

\begin{definition}
 Let $R, S$ be as before. Let $U$ be a subgroup of $R^{\times}$ such that $-1 \in U$.
 We say that $\Gamma(R,S)$ is a $U$-unitary Cayley graph if $S$ is stable under the left and right action of $U$; namely, $USU = S$. Note that since $-1 \in U$, this condition automatically implies that $S$ is symmetric.
\end{definition}

\begin{rem}
    As explained in \cite[Proposition 2.1]{minavc2024complete}, an element $r \in R$ is left or right invertible if and only if it is invertible. Therefore, there is no left/right ambiguity in the definition of $U$ and $R^{\times}.$
\end{rem}

\begin{rem}
    When $U=R^{\times}$, we will use the term $R^{\times}$-unitary Cayley graph and gcd-graph interchangeably. 
\end{rem}
We define the following relation on $R$: we say that $x \sim_{U} y$ if $x=u_1 y u_2$ where $u_1, u_2 \in U.$ We can see that this is an equivalence relation. Furthermore, if we denote by $I_x$ (respectively $I_y$) the two-sided ideal generated by $x$ (respectively $y$) then $I_x = I_y.$ Here, we recall that the two-sided ideal $I_x$ is the set of elements of the forms $\sum_{i=1}^n a_{i} xb_{i}$ where $a_{i}, b_{i} \in R.$ 
\begin{rem}
We remark that, in contrast to the commutative case, if $I_x=I_y$, then it is not necessarily true that $x \sim_{R^{\times}} y.$ In fact, if $R=M_n(F)$, then $R$ has no proper two-sided ideals. Consequently, for each $A \neq 0$ in $M_n(F)$, the two-sided ideal generated by $A$ is $M_n(F)$ itself. 
\end{rem}
Let $\mathcal{K} = \{K_1, K_2, \ldots, K_{m} \}$ be the elements of the  double quotient  $U \backslash R \slash U$. In other words, $\mathcal{K}$ is precisely the equivalence classes of $R$ with respect to $\sim_{U}.$ By the definition, we have the following criterion for a graph to be $U$-unitary. 
\begin{prop}
The following conditions are equivalent.
\begin{enumerate}
    \item $\Gamma(R,S)$ is a $U$-unitary graph. 
    \item $S$ is a disjoint union of some orbits $K_i$.
\end{enumerate}
\end{prop}
\begin{cor}
    If $\Gamma(R,S)$ is a $U$-unitary Cayley graph, then its complement graph $\Gamma(R,S)^c$ is also a unitary Cayley graph. 
\end{cor}
\begin{proof}
    By definition, $\Gamma(R, S)^c$ is precisely $\Gamma(R, S')$ where $S' = R \setminus (S \cup \{0\}) = \mathcal{K} \setminus (S \cup \{0\}).$ Since $S$ and $\{0\}$ are both unions of some orbits $K_i$'s, the same holds true for $S'.$ 
\end{proof}

\subsection{Graph-theoretic properties of $U$-unitary Cayley graphs}
\subsubsection{Connectedness}
In this section, we answer the following question: when is a $U$-unitary Cayley graph connected? To do so,  we first introduce the following convention: for each $1 \leq i \leq m$, let $I_i$ be the ideal generated by an element of $ x \in K_i$ (by definition of $K_i$, $I_i$ is independent of the choice of $x$). Let $\ell(K_i)$ be the smallest number such that every element $y$ in $I_i$ be written in the form $y = \sum_{i=1}^{\ell(K_i)} a_{i} x b_i$ where $a_i, b_i \in R$. We note that since $R$ is finite, such a number always exists. Furthermore, if $R$ is commutative, then $\ell(K_i)$ is either $0$ or $1.$  We now state the main result for the connectedness of a $U$-unitary Cayley graph.

\begin{prop} \label{prop:connectedness}
Let $\Gamma(R,S)$ be a $U$-unitary Cayley graph. Suppose that $\Gamma(R,U)$ is connected. Then the following conditions are equivalent. 
    \begin{enumerate}
        \item $\Gamma(R,S)$ is connected. 
        \item $R=\sum_{i} I_i$ where the sum is over all ideals $I_i$ such that $ K_i \subset S.$
    \end{enumerate}
    Furthermore, if one of these conditions is satisfied, then 
     \[ {\rm diam}(\Gamma(R,S)) \leq {\rm diam}(\Gamma(R,U))^2 \sum_{s=1}^t \ell(K_{i_s}) .\] 
    Here $t$ is the smallest positive integer in which there exists $i_1, i_2, \ldots, i_t$ such that $R=\sum_{s=1}^t I_{i_s}$ and $S \cap K_{i_s} \neq \emptyset.$
\end{prop}
\begin{proof}
    The proof for this statement is a slight modification of the one given for the commutative case in \cite[Proposition 3.10]{nguyen2025supercharacters}. We remark, however, that in this case the upper bound for the diameter of $\Gamma(R,S)$ is slightly larger due to the non-commutativity of $R.$

    We first show that $(1) \implies (2).$ Indeed, because $\Gamma(R,S)$ is connected, we can find a $\Z$-linear combination: $1 = \sum_{i} a_i s_i,$ 
    where $a_{i} \in \Z$ and $s_{i} \in S.$ By definition of $I_i$, we know that $\sum_{i} a_i s_i \in \sum_{i} I_i$. This shows that $1 \in \sum_{i} I_i$, which implies that $R=\sum_{i} I_i.$

    Conversely, suppose that $(2)$ holds. We will show that $\Gamma(R,S)$ is connected. In fact, let $r \in R.$ Since $R=\sum_{s=1}^t I_{i_s}$, we can write  
    \begin{equation} \label{eq:combination}
    r= \sum_{s=1}^t \left(\sum_{j=1}^{\ell(K_{i_s})} a_{i_s j} x_{i_s} b_{i_s j} \right),
    \end{equation}
     where $x_{i_s} \in K_{i_s}$, and  $a_{i_sj}, b_{i_s j} \in R.$ Furthermore, since $\Gamma(R,U)$ is connected, each $a_{i_s j}$ and $b_{i_s j}$ can be written as a sum of at most $\text{diam}(\Gamma(R,U))$ elements in $U$. Consequently, each term $a_{i_s j} x_{i_s} b_{i_s j}$ can be written as a $\Z$-linear combinations of at most $\text{diam}(\Gamma(R,U))^2$ elements in $K_{i_s}.$ From \cref{eq:combination}, we conclude that 
     \[ \text{diam}(\Gamma(R,S)) \leq \text{diam}(\Gamma(R,U))^2 \sum_{s=1}^t \ell(K_{i_s}).
     \qedhere \] 
\end{proof}

\begin{cor} \label{cor:matrix-connected}
    Let $R=M_n(F)$ where $n>1$ and $F$ be a finite field. Let $U \subset R^{\times} = \text{GL}_n(F)$ such that $\Gamma(R, U)$ is connected. Let $\Gamma(R,S)$ be a $U$-unitary Cayley graph such that $S$ is not an empty set. Then $\Gamma(R,S)$ is connected. 
\end{cor}

\begin{proof}
This follows from \Cref{prop:connectedness} and the fact that the only two-sided ideals of $M_n(F)$ are $0$ and $M_n(F).$
\end{proof}

\begin{rem} \label{rem:sum-invertible}
    By \cite{invertible_matrices}, every square matrix of size $n \times n$ with $n>1$ is the sum of two invertible matrices. Therefore, the graph $\Gamma(M_n(F), GL_n(F))$ is always connected.  The above corollary shows that all non-empty $GL_n(F)$-unitary Cayley graphs are connected. 
\end{rem}

It turns out that every non-empty $SL_n(F)$-graph is connected. 
\begin{prop}
    The Cayley graph $\Gamma(M_n(F), SL_n(F))$ is connected. 
\end{prop}

\begin{rem}
    We remark that since we only work with undirected graphs, we assume implicitly here that $-I_n \in SL_n(F).$ This happens only if $\text{char}(F)=2$ or $n$ is even. 
\end{rem}
\begin{proof}
    By the proof of \cref{prop:similar-sln}, for two singular matrices $A,B$, $A \sim_{GL_n(F)} B$ if and only if $A \sim_{SL_n(F)} B.$ Furthermore, since every matrix is a sum of matrices with exactly one non-zero element, it is sufficient to show that for each $a \in F$, the following can be written as a sum of elements in $SL_n(F)$ 
\[A = a \oplus 0_{n-1} = \begin{bmatrix} a & 0 \\ 0 & 0 \end{bmatrix}. \]
Here $0_{n-1}$ is the $(n-1) \times (n-1)$ zero matrix in $M_{n-1}(F).$ Let $X$ and $Y$ be as follows. 

\[ X = \begin{bmatrix} 1+a & 1 \\ -1 & 0 \end{bmatrix}, X' = X \oplus I_{n-2}, \]
\[ Y = \begin{bmatrix} 1 & 1 \\ -1 & 0 \end{bmatrix}, X' = Y' \oplus I_{n-2}, \]
Then $A=X'+ (-Y')$ and $\det(X')=\det(-Y')=1.$ We remark that we implicitly use the assumption that $\det(-I_n)=1.$
\end{proof}

\subsubsection{Primeness}
In this section, we study the primeness of $U$-unitary Cayley graphs. First, we need to recall this definition.

A subset $X$ in a graph $G$ is called a homogeneous set if every vertex in $V(G) \setminus X$ is adjacent to either all or none of the vertices in $X$. By definition, if $X=V(G)$ or $|X|=1$ then $X$ is a homogeneous set-- it is called a trivial homogeneous set. Otherwise, a homogeneous set $X$ with $2 \leq X < |V(G)|$ is called non-trivial. The graph $G$ is said to be prime if it does not contain any non-trivial homogeneous sets. We note that, by definition, a set $X$ is homogeneous in $G$ if and only if $X$ is homogeneous in the complement $G^c$ of $G.$ Additionally, we also note that the notion of a homogeneous set generalizes the notion of a connected component; namely, a connected component of a graph is always a homogeneous set. For this reason, when studying prime graphs, we can safely assume that $\Gamma(R,S)$ and its complement are both connected.

As explained in previous works such as \cite{chudnovsky2024prime, nguyen2025supercharacters, nguyengcd2026}, the existence of a homogeneous set on a Cayley graph requires some rather strong conditions on the generating set.  In particular, in \cite{nguyen2025supercharacters}, we show that if a $U$-unitary Cayley graph over a commutative ring is not prime, then there exists a proper ideal $I$; namely, $I \neq 0$ and $I \neq R$, such that $I$ is a homogeneous set. This statement can be generalized to the non-commutative setting.

\begin{prop} \label{prop:homogeneous_ideal}
Suppose that $\Gamma(R,U)$ is connected. Suppose further that $\Gamma(R,S)$ is both connected and anti-connected. Then, the following conditions are equivalent. 
\begin{enumerate}
    \item $\Gamma(R,S)$ is not a prime graph. 
    \item There exists a proper two-sided ideal $I$ in $R$ such that $I$ is a homogeneous set in $\Gamma(R,S).$
\end{enumerate}
\end{prop}
\begin{proof}
By definition, $(2) \implies (1).$ Let us show that $(1) \implies (2).$ By \cite[Theorem 3.4]{chudnovsky2024prime}, if $I$ is a maximal non-trivial homogeneous set of $\Gamma(R,S)$ containing $0$, then $I$ is a subgroup of $(R, +).$  We claim that $I$ is a left ideal in $R$ as well. For each $u \in U$,  the left multiplication by $u$ is an automorphism of $\Gamma(R,S).$ Consequently, $uI$ is also a homogeneous set. Since $0 \in I \cap uI$, we know that  $I \cup uI$ is also a homogeneous set (see \cite[Lemma 3.1]{chudnovsky2024prime}). By the maximality of $I$, we must have $uI=I$. We conclude that $I$ is stable under the left action of $U.$ We now show that if $r \in R$, then $rI \subset I$. In fact, since $\Gamma(R,U)$ is connected, we can write
$r = \sum_{i=1}^d m_i u_i,$ where $m_i \in \Z$ and $u_i \in U.$ For each $h \in I$, we have  $rh = \sum_{i=1}^d m_i (u_i h).$
 Since $u_i h \in I$ and $I$ is a subgroup of $(R, +)$, we conclude that $rh \in I.$ This shows that $rI \subset I$ for all $r \in R$. Therefore, $I$ is a left ideal in $R$. An identical argument shows that $I$ is also a right ideal in $R.$ We conclude that $I$ is a two-sided ideal in $R.$
\end{proof}
Let us discuss some corollaries of \cref{prop:homogeneous_ideal}.
\begin{cor}
    Let $R=M_n(F)$ with $n>2$ and $F$ be a finite field. Let $U \subset R^{\times} = \text{GL}_n(F)$ such that $\Gamma(R, U)$ is connected. Let $\Gamma(R,S)$ be a $U$-unitary Cayley graph such that $S \neq \emptyset$ and $S \neq M_n(F) \setminus \{0\}$ (equivalently $\Gamma(R,S)$ is not a complete or empty graph). Then $\Gamma(R,S)$ is prime. In particular, if $U=GL_n(F)$ or $U=SL_n(F)$, then $\Gamma(R,S)$ is always prime. 
\end{cor}

\begin{proof}
    By \cref{cor:matrix-connected}, we know that $\Gamma(R,S)$ is both connected and anti-connected. Furthermore, $M_n(F)$ has no proper two-sided ideals. Therefore, \cref{prop:homogeneous_ideal} shows that $\Gamma(R,S)$ is prime. 
\end{proof}

In the case $S=U$, we have a rather strict condition when $\Gamma(R,U)$ is not prime. To discuss this restriction,  we recall that the Jacobson radical $\Rad(R)$ of $R$ is the intersection of all left maximal ideals in $R$ (see \cite[Chapter 4.3]{pierce1982associative}). It is known that $\Rad(R)$ is a two-sided ideal in $R$. 

\begin{prop}
    Let $I$ be a left (or right) ideal, which is also a homogeneous set in $\Gamma(R,U).$ Then $I + U \subset U.$ Furthermore, $I \subset \Rad(R).$
\end{prop}

\begin{proof}
    Since $U \subset R^{\times}$, $U \cap I = \emptyset.$  Let $u$ be any element in $U$. Then $(u,0)$ is an edge in $\Gamma(R,U).$ Since $0 \in I$ and $I$ is homogeneous, $(u, -x)$ is also an edge for each $x \in I.$ By definition, this shows that $u+x \in U.$ Hence $I+U\subset U$.
    The fact that $I \subset \Rad(R)$ follows from \cite[Chapter 4.3]{pierce1982associative}). 
\end{proof}
We have the following immediate corollary. 
\begin{cor} \label{cor:prime}
    If $R$ is semisimple, that is, $\Rad(R)=0$, then there is no two-sided ideal $I$ such that $I$ is homogeneous in $\Gamma(R,U).$ In particular, if $\Gamma(R,U)$ is connected and anti-connected, then $\Gamma(R,U)$ is prime. 
\end{cor}
\subsubsection{Unitary Cayley graphs}
In this section, we focus on the case $S=R^{\times}$. In this case, the associated graph is denoted by $G_R$ and is well-known as the unitary Cayley graph of $R$ (see \cite{unitary, chen2022unitary}). This type of graph has a rich history; we can trace its roots in the work of Evans and Erdős in \cite{erdos1989representations}. In \cite{chudnovsky2024prime, minavc2024complete, nguyengcd2026}, we provide a complete classification of finite commutative rings $R$ such that $G_{R}$ is connected/prime. In this section, we give a similar answer for all finite unital rings. 

 We note that, by \cite[Proposition 4.30]{chudnovsky2024prime}, $\Rad(R)$ is a homogeneous set in $G_R.$ By the same argument as in \cite[Corollary 4.2]{chudnovsky2024prime}, we have the following isomorphism 
\[ G_R \cong G_{R^{\s}} * E_n, \]
here $R^{\s} = R/\Rad(R)$ is the simplification of $R$, $E_n$ is the empty graph on $n = |\Rad(R)|$ vertices, and $*$ denotes the wreath product of two graphs (see \cite[Definition 2.5]{chudnovsky2024prime} for the definition of the wreath product of graphs). This isomorphism shows that $G_{R}$ is connected if and only if $G_{R^{\s}}$ is connected. Furthermore, if $G_{R}$ is prime then $\Rad(R)=0$; namely $R=R^{\s}$. Therefore, we can assume from now on that $R=R^{\s}.$ In this case, the Artin-Wedderburn theorem implies that $R^{\s}$ is a product of local semisimple rings  
\begin{equation} \label{eq:product_rings}
R^{\s} = \prod_{i=1}^s R_i \times \prod_{i=1}^r M_{d_i}(F_i).
\end{equation}
Here $R_i$ is a finite field such that $2 \leq |R_1| \leq |R_2|\leq \cdots \leq |R_s|$. Additionally, $d_i \geq 2$, and $F_i$ is a finite field. We can then see that $G_{R^{\s}}$ is a direct product of unitary Cayley graphs 
\[ G_{R^{\s}} = \prod_{i=1}^s G_{R_i} \times \prod_{i=1}^r G_{M_{d_i}(F_i)} = \prod_{i=1}^s K_{|R_i|} \times \prod_{i=1}^r G_{M_{d_i}(F_i)}.\]
Here, $K_n$ is the complete graph on $n$ vertices.  We have the following proposition, which is a direct generalization of \cite[Lemma 4.33]{chudnovsky2024prime},   \cite[Theorem 3.6]{nguyengcd2026}, and \cite[Theorem 1.1]{maimani2010rings}.
\begin{prop} \label{prop:unitary-connected}
    $G_{R^\s}$ (and hence equivalently $G_{R}$) is  connected if and only if in the above decomposition, there is at most one $i \in \{1, \ldots, s\}$ such that $|R_i|=2.$
\end{prop}

\begin{proof}
    If there are more than two $i$ such that $|R_i|=2$, then the direct product contains a copy of $K_2 \times K_2$ which is not connected. Therefore, $G_{R^{\s}}$ is also not connected. Conversely, suppose that there is at most one $i$ such that $|R_i|=2.$ By our ordering, $|R_1|=2$ and $|R_k|>2$ for each $2 \leq k \leq s. $ For these $k$, each graph $G_{R_{k}}$ is connected and non-bipartite. Similarly, for $1 \leq i \leq r$, $G_{M_{d_i}(F_i)}$ is also connected and non-bipartite (see \cite[Proposition 3.5]{minavc2024complete}). By \cite[Corollary 5.10]{hammack2011handbook}, $G_{R^\s}$ is connected. 
\end{proof}

We now classify $R$ such that $G_{R}$ is prime. As explained above, if $G_{R}$ is prime, then $R$ is necessarily semisimple; namely $R=R^{\s}.$ Additionally, $G_{R}$ must also be connected and anti-connected. By \cref{prop:unitary-connected} if $G_{R^\s}$ is connected, there is at most one $i$ such that $|R_i|=2$ in the Artin-Wedderburn decomposition of $R.$ For anti-connected, we have the following simple observation. 

\begin{prop} \label{prop:unitary-anti-connected}
    $G_{R^{\s}}$ is not anti-connected if and only if $R^{\s}$ is a field.  
\end{prop}

\begin{proof}
    If $R^{\s}$ is a field then $G_{R^{\s}}$ is a complete graph. Therefore, its complement is not connected. Conversely, assume that $R^{\s}$ is not a field. There are two cases to consider. 

\textbf{Case 1.} $R^{\s}$ is a product of two rings; say $R=R_1 \times R_2.$ We claim that for each $(r_1, r_2) \in R_1 \times R_2$, there is a walk in $G_{R^{\s}}^c$ between $(0,0)$ and $(r_1, r_2)$ (by a translation, this shows that there is a walk between any two vertices in $G_{R^{\s}}$. If one of $r_1$ or $r_2$ is not a unit then $(0,0)$ and $(r_1, r_2)$ are adjacent in $G_{R^\s}^c.$ Now, suppose that $r_1, r_2$ are both units. Then have the following walk 
\[ (0,0) \to (r_1, 0) \to (r_1, r_2). \]

\textbf{Case 2.} $R^{\s} = M_n(F)$ for some $n>1$ and $F$ is a field. Let $A=[v_1 \quad v_2 \quad \ldots \quad v_n]$ be a matrix formed by column vectors $\{v_1, v_2, \ldots, v_n\}.$ We claim that there is a walk between $0$ and $A.$ In fact, let $A_1=[v_1 \quad v_2 \quad \ldots v_{n-1} \quad 0]$. Then both $A_1$ and $A-A_1$ are not invertible. Consequently, we have the following walk in $G_{R^{\s}}^c$ 
\[ 0 \to A_1 \to A. \] 
Since this is true for all $A$, we conclude that $G_{R^\s}^c$ is connected. 
\end{proof}
We are now ready to state and prove the main theorem about the primeness of $G_{R^{\s}}.$
\begin{prop} \label{prop:prime-unitary}
$G_R$ is  prime if and only if the following conditions are satisfied
\begin{enumerate}
    \item $R=R^{\s}.$
    \item Let 
    \[ R^{\s} = \prod_{i=1}^s R_i \times \prod_{i=1}^r M_{d_i}(F_i),\]
    be the decomposition of $R^{\s}$ into products of fields and matrix rings as in \cref{eq:product_rings}. 
    Then there is at most one $1 \leq i \leq s$ such that $|R_i|=2.$
    \item $R^{\s}$ is not a field ($R^{\s}$ is a field if and only if $s=1$ and $r=0$). 
\end{enumerate}
\end{prop}

\begin{proof}
    We have shown that these conditions are necessary in \cref{prop:unitary-connected} and \cref{prop:unitary-anti-connected}. Let us show that they are sufficient. Condition $(1)$ implies that $G_{R}=G_{R^{\s}}$. Additionally, conditions $(2)$ and $(3)$ imply that $G_{R}$ is connected and anti-connected. By \cref{cor:prime} and \cref{prop:unitary-connected}, we conclude that $G_{R}$ is a prime graph. 
\end{proof}

\subsubsection{Primeness and singularity of the adjacency matrix}
In \cite[Section 4.3]{nguyen2025supercharacters}, we study a quite interesting relationship between the primeness of a $U$-unitary graph and the singularity of its adjacency matrix. More precisely, we show that if $R$ is commutative and $U \subset R^{\times}$ such that $\Gamma(R,U)$ is connected, anti-connected, and not prime, then the adjacency matrix of $\Gamma(R,U)$ is singular. Equivalently, $0$ is an eigenvalue of $\Gamma(R,U) $(see \cite[Proposition 4.16]{nguyen2025supercharacters}). We observe that this statement holds for any finite ring $R$. 

\begin{prop} \label{prop:0-eigenvalue}
    Suppose that $\Gamma(R,U)$ is both connected and anti-connected. Suppose further that $\Gamma(R,U)$ is not prime. Then $0$ is an eigenvalue of $\Gamma(R,U).$
\end{prop}

\begin{proof}
By \cref{prop:homogeneous_ideal}, there exists a two-sided ideal $I$ such that $I$ is a homogeneous set in $\Gamma(R,U).$ Furthermore, since $I \cap U = \emptyset$, $\Gamma(R,U)$ is isomorphic to the wreath product $\Gamma(R/I, U) *  E_{|I|}.$ By \cite[Theorem 3.3]{minac2025joins}, $0$ is an eigenvalue of $\Gamma(R,S).$ 
\end{proof}
We are interested in the converse of \cref{prop:0-eigenvalue}.
\begin{question} \label{question:prime-0}
If $ \Gamma(R,U)$ is prime, is it true that $0$ is not an eigenvalue of $\Gamma(R,U)$?  
\end{question}

Here, we provide a partial answer to this question for the case of unitary Cayley graphs. 

\begin{prop}
    Suppose that $G_{R}$ is prime. Then $0$ is not an eigenvalue of $G_R.$
\end{prop}
\begin{proof}
    By \cref{prop:prime-unitary}, if $G_R$ is prime, then $R$ is semisimple; namely 
    \[ R = \prod_{i=1}^s R_i \times \prod_{i=1}^r M_{d_i}(F_i), \]
    where $R_i, F_i$ are fields and $d_i \geq 2.$ Therefore 
    \[ G_{R} \cong \prod_{i=1}^s K_{|R_i|} \times \prod_{i=1}^r G_{M_{d_i}(F_i)}.\]
Here, all products are the direct product.    An eigenvalue of $G_{R}$ is therefore of the form $\prod_{i=1}^s \lambda_{i} \prod_{i=1}^r \lambda_{i}' $ where $\lambda_i$ is an eigenvalue of $K_{|G_i|}$ and $\lambda_i'$ is an eigenvalue of $G_{M_{d_i}(F_i)}.$ Since $K_{|G_i|}$ is a complete graph, $\lambda_i \neq 0.$ Similarly, by \cite[Theorem 1.1]{chen2022unitary}, $\lambda_i' \neq 0$ as well. Therefore, all eigenvalues of $G_{R}$ are not $0.$
\end{proof}

\section{Spectra theory of $U$-unitary Cayley graphs \\ over a finite symmetric Frobenius algebra} \label{sec:spectra}

\subsection{Symmetric Frobenius algebras}
Let $S$ be a finite commutative ring. Let $R$ be an $S$-algebra; namely, there is a ring homomorphism $f\colon S \to R$ such that $f(S)$ is contained in the center of $R.$ To avoid lengthy notations, we will often identify $S$ and its image $f(S)$ in $S.$ 

\begin{definition}
We say that $R$ is a symmetric Frobenius $S$-algebra if there exists a $S$-module morphism $\psi_{R,S}\colon R \to S$ satisfying the following conditions 
\begin{enumerate}
    \item $\psi_{R,S}(r_1r_2) =\psi_{R,S}(r_2 r_1)$. 
    \item The kernel of $\psi_{R,S}$ does not contain any non-zero left ideal in $R.$

\end{enumerate}
\end{definition}
\begin{prop} \label{prop:composition-frobenius}
Suppose that $R$ is a symmetric Frobenius $S$-algebra, and $S$ is a symmetric Frobenius $T$-algebra. Then $R$ is a symmetric Frobenius $T$-algebra. 
\end{prop}
\begin{proof}
    Let $\psi_{R,S}\colon R \to S$ (respectively $\psi_{S, T}\colon S \to T$) be the morphism that makes $R$ (respectively $S$) a symmetric Frobenius $S$-algebra (respectively $T$-algebra). Let $\psi_{R, T}$ be the composition of $\psi_{S,T} \circ \psi_{R,S}.$ By definition $\psi_{R,T}\colon R \to T$ is $T$-linear and symmetric. We claim that its kernel does not contain any non-trivial left ideal in $R.$ In fact, suppose to the contrary, it is not the case. Then, there exists $x \in R$ such that $\psi_{R,T}(rx)=0$ for all $r\in R.$ For each $s \in S$, $sr \in R$ and hence $\psi_{R,T}(srx)=0.$ By definition 
    \[ 0 = \psi_{R,T}(srx) = \psi_{S, T} (\psi_{R,S}(srx)) = \psi_{S,T}(s \psi_{R,S}(rx)) . \]
    Since $\psi_{S,T}$ does not contain a non-trivial left ideal, we conclude that $\psi_{R,S}(rx)=0.$ Since this is true for all $r \in R$, $\ker(\psi_{R,S})$ contains the left ideal generated by $x$--which is a contradiction. 
\end{proof}

\begin{prop} \label{prop:matrix-frobenius}
    Let $R$ be a finite commutative ring and $n$ a positive integer. Then $M_n(R)$ is a symmetric Frobenius $R$-algebra under the trace map.  
\end{prop}
\begin{proof}
    Let $T\colon M_n(R) \to R$ be the trace map. It is known that $T$ is symmetric; namely $T(AB)=T(BA)$ for all $A,B \in M_n(R).$ We now show that $T$ is non-degenerate. In fact, suppose to the contrary that $\ker(T)$ contains a non-zero left ideal in $M_n(R).$ Then, there exists a non-zero matrix $A \in M_n(R)$ such that $T(BA)=0$ for all $B \in M_n(R).$ Let $E_{ij}$ be the matrix whose $(i,j)$-position is $1$ and $0$ everywhere else. We have 
    \[ 0= T(E_{ij}A)=A_{ji}.\]
    Since this is true for all $1 \leq i,j \leq n$, we must have $A=0$, which is a contradiction. 
\end{proof}
We have the following corollary. 
\begin{cor} \label{cor:matrix-symmetric}
    Let $S, R$ be commutative rings such that $R$ is a symmetric  Frobenius $S$-algebra. Then $M_n(R)$ is also a symmetric  Frobenius $S$-algebra. In particular, if $F$ is a finite field whose base field is $\F_p$, then $M_n(F)$ is a symmetric Frobenius $\F_p$-algebra. 
\end{cor}

\begin{proof}
    The first part follows from  \cref{prop:composition-frobenius} and \cref{prop:matrix-frobenius}. The second part follows from the first part and the fact that $F:=\F_q$ is a symmetric Frobenius $\F_p$-algebra under the classical trace map $\text{Tr}\colon \F_q \to \F_p.$
\end{proof}

\subsection{Spectral description of $U$-unitary Cayley graphs}
Suppose that $R$ is a symmetric Frobenius $\Z/n$-algebra for some $n >1.$ Let $\psi$ be the associated $\Z/n$-linear functional $\psi\colon R \to \Z/n$.  Let $\zeta_n$ be a fixed primitive root of unity in $\C$. Let $\chi\colon R \to \C$ be the character defined by $\chi(s)=\zeta_n^{\psi(s)}.$ For each $r \in R$, let $\chi_r$ be the character defined by $\chi_r(s)=\chi(rs).$ The map 
\[ \Phi\colon R \to \Hom(R, \C^{\times}), \] 
sending $r \mapsto \chi_r$ is a group homomorphism (with respect to the additive structure on $R$). Since $\psi$ is non-degenerate, $\Phi$ is injective. However, since $R$ is finite, it is an isomorphism as well. In other words, $\chi$ is a generating character for the dual group $\Hom(R, \C^{\times})$ (see \cite{honold2001characterization}). We remark that since $\psi$ is symmetric, for each $r,s \in R$, $\chi_{r}(s)=\chi_{s}(r).$

By the circulant diagonalization theorem (see \cite{kanemitsu2013matrices}), we know that the eigenvalues of the Cayley graph $\Gamma(R,S)$ are given by the multiset $\{\lambda_r\}_{r \in R}$ where 
\[ \lambda_r = \sum_{s \in S} \chi_{r}(s).\]
Next, we will show that when $\Gamma(R,S)$ is $U$-unitary, there is a quite elegant description of the above sum by certain supercharacter theory on $R.$
To do so, we first need to recall the definition of a supercharacter theory. 
\begin{definition} (see \cite[Definition 2.1]{nguyen2025supercharacters}) \label{def:supercharacter}
    Let $G$ be a finite abelian group. Let $\mathcal{K}=\{K_1, K_2, \ldots, K_m\}$ be a partition of $G$ and $\mathcal{X} =\{X_1, X_2, \ldots, X_m \}$ a partition of the dual group $\widehat{G}=\Hom(G, \C^{\times})$ of characters of $G$. We say that $(\mathcal{K}, \mathcal{X})$ is a supercharacter theory for  $G$ if the following conditions are satisfied 
    \begin{enumerate}
        \item $\{0 \} \in \mathcal{K};$
        \item $|\mathcal{X}| = |\mathcal{K}|;$
        \item For each $X_i \in \mathcal{X}$, the character sum 
        \[ \sigma_i = \sum_{\chi \in X_i} \chi\]
        is constant on each $K \in \mathcal{K}$;
    \item For a  fixed $\chi \in \mathcal{X}$ the sum $\sum_{k \in K_i} \chi(k)$ does not depend on the choice of $\chi \in X.$ 

    \end{enumerate}
\end{definition}
We now show that each $U$ induces a supercharacter theory on $R$ (this generalizes \cite[Theorem 4.1]{nguyen2025supercharacters} for the case $R$ is commutative). More precisely

\begin{thm} \label{prop:supercharacter-U}
let $\mathcal{K}=\{K_1, K_2, \ldots, K_m\}$ be the double quotient $U \backslash R \slash U.$ Additionally, let $\mathcal{X} = \{X_1, X_2, \ldots, X_m\}$ be the partition of the character group of $R$ defined by 
    \[ X_i = \{\chi_{x} \mid x \in K_i \}.\]
Then the pair $(\mathcal{K}, \mathcal{X})$ is a symmetric supercharacter theory for $R.$ Furthermore, $(\mathcal{K}, \mathcal{X})$ satisfies Condition $4$ in Definition \ref{def:supercharacter}.\end{thm}
\begin{proof}
    The first two conditions are clear from the definition of $\mathcal{K}$ and $\mathcal{X}.$ Let us prove the third condition. Let 
    \[ \sigma_i = \sum_{x \in K_i} \sigma_x.\]
Suppose that $y \sim_{U} z$, we will show that $\sigma_i(y)=\sigma_i(z).$ In fact, by definition, $y=u_1 z u_2$ for some $u_1, u_2 \in U.$ We then have 
\begin{align*}
\sigma_i(y) &= \sum_{x \in K_i} \sigma_x(y) = \sum_{x \in K_i} \sigma_x(u_1 z u_2) = \sum_{x \in K_i} \chi(xu_1 zu_2) \\ 
&= \sum_{x \in K_i} \chi(u_2 xu_1 z) = \sum_{t \in u_2 K_i u_1} \chi(tz) = \sum_{t \in u_2 K_i u_1} \chi_t(z) = \sum_{x \in K_i} \chi_x(z) = \sigma_i(z). 
\end{align*}
Here, the third equality follows from the fact that $\chi(rs)=\chi(sr)$. The last equality follows from the fact that $u_2K_iu_1 = K_i.$ We conclude that the pair $(\mathcal{K}, \mathcal{X})$ satisfies the third condition of Definition \ref{def:supercharacter}.  Finally, the last condition of Definition \ref{def:supercharacter} can be obtained by an almost identical argument as above. 
\end{proof}

For each $1 \leq i, j \leq n$ we define, as in \cite{nguyen2025supercharacters},  the following notation 
\[ \Omega_{ji} = \sum_{k \in K_j} \chi_{x_i} (k), \]
here $x_i \in K_i$. (By \cref{prop:supercharacter-U}, $\Omega_{ji}$ does not depend on the choice of $x_i.$) 

\begin{rem}

In \cite[Proposition 4.5]{nguyen2025supercharacters}, we describe $\Omega_{ji}$ via certain generalized Ramanujan sums when $R$ is commutative. Unfortunately, it is not clear to us how to do so in the non-commutative case. In the special case of matrix rings, \cite{chen2022unitary} provides a rather elegant calculations of these sums when $R=M_n(F)$ and $U=GL_n(F).$
    
\end{rem}

We are now ready to state the main theorem about spectra of $U$-unitary Cayley graphs. 
\begin{thm} \label{prop:eigenvalues}
Let $\Gamma(G,S)$ be an $U$-unitary Cayley graph. Then,  the spectrum of $\Gamma(G,S)$ is the multiset $\{[\lambda_i]_{|K_i|} \}_{1 \leq i \leq m}$ where 
    \[\lambda_i = \sum_{K_j \subset S} \Omega_{ji}.\]
    Consequently, $\Gamma(G,S)$ has at most $m$ distinct eigenvalues where $m=|U \backslash R \slash U|.$
\end{thm}
\begin{proof}
    For each $r \in R$, let $\lambda_r$ be the eigenvalue associated with $r$ under the circulant diagonalization theorem. We then have 
    \[ \lambda_{x_i} = \sum_{s \in S} \chi_{x_i}(s) = \sum_{K_j \subset S} \left(\sum_{k \in K_j} \chi_{x_i}(k)\right) = \sum_{K_j \subset S} \Omega_{ji}. \]
    We also know that if $x_i \sim_{U} x_i'$, then $\lambda_{x_i}=\lambda_{x_i'}.$ Therefore, each $\lambda_{x_i}$ occurs $|K_i|$-times. 
\end{proof}
\subsection{The case of matrix rings}
In this section, we apply \cref{prop:eigenvalues} to some $U$-unitary Cayley graphs over a matrix ring. We will show, in particular, that the upper bound in \cref{prop:eigenvalues} can be strict in several cases. First, we study the case of gcd-graphs over $M_n(F).$
\begin{cor} \label{cor:number-of-eigs-matrix}
Let $R=M_n(F)$ where $F$ is a finite field.  Let $\Gamma(R,S)$ be a $R^{\times}$-unitary Cayley graph. Then $\Gamma(R,S)$ has at most  $(n+1)$ distinct eigenvalues. 
\end{cor}
\begin{proof}
By \cref{cor:matrix-symmetric}, we know that $M_n(F)$ is a symmetric Frobenius $\F_p$-algebra where $\F_p$ is the base field of $F.$  By the theory of row and columns operations, for two matrices $A$ and $B$, $A \sim_{GL_n(F)} B$ if and only if $\rank(A)=\rank(B).$ Therefore, $|GL_n(F) \backslash M_n(F) \slash GL_n(F)|$ has exactly $n+1$ equivalence classes. Therefore, by \cref{prop:eigenvalues}, $\Gamma(R,S)$ has at most $n+1$ distinct eigenvalues. 
\end{proof}
\begin{rem}
    The upper bound in \cref{cor:number-of-eigs-matrix} is optimal. In fact, in \cite[Theorem 1.1]{chen2022unitary}, the authors show that the unitary Cayley graph on $M_n(F)$ has exactly $n+1$ distinct eigenvalues. It would be interesting to see whether their calculations could be generalized to all $R^{\times}$-unitary Cayley graphs over $M_n(F).$
\end{rem}
We now discuss another result which shows the power of \cref{cor:number-of-eigs-matrix}. Let $R=GL_n(F)$ and $U=SL_n(F)$--the set of all invertible matrices with determinant $1.$
\begin{prop} \label{prop:similar-sln}
    $|SL_n(F) \backslash M_n(F) \slash SL_n(F)|=n+|F|-1.$
\end{prop}
\begin{proof}
    For $A,B \in M_n(F)$ such that $A \sim_{SL_n(F)} B$, then $\det(A)=\det(B).$ Therefore, if $A,B$ are invertible and $\det(A) \neq \det(B)$, then $A$ and $B$ are not equivalent. This shows that 
    \[ |SL_n(F) \backslash GL_n(F) \slash SL_n(F)|=|F|^{\times}=|F|-1.\]

    On the other hand, we claim that if $A \sim_{GL_n(F)} B$ and $A,B$ are not invertible, then $A \sim_{SL_n(F)} B.$ Let $r=\rank(A)=\rank(B).$ Then $r<n$. Let $I_r$ be the identity matrix of size $r\times r$ and $\widehat{I_r}$ be the following $n\times n$-matrix 
    \[ 
\widehat{I_r} = I_r\oplus 0 = \begin{bmatrix} I_r & 0 \\ 0 & 0 \end{bmatrix}. \]
By row and column operations, we know that there exists $P,Q \in GL_n(F)$ such that $A=P \widehat{I_r} Q.$ Let $P'$ (respectively $Q'$) be the new matrix obtained by multiplying the last column of $P$ (respectively the last row of $Q$) by $\dfrac{1}{\det(P)}$ (respectively $\dfrac{1}{\det(Q)}$). Then, we still have 
\[ A = P\widehat{I_r} Q = P' \widehat{I_r} Q'.\]
Furthermore, $\det(P')=\det(Q')=1.$ Therefore, $A \sim_{SL_n(F)} \widehat{I_r}.$ Similarly, $B \sim_{SL_n(F)} \widehat{I_r}$, and hence $A \sim_{SL_n(F)} B.$ This shows that the rank map gives an isomorphism between $SL_n(F) \backslash [M_n(F)-GL_n(F)] \slash SL_n(F)$ and the set $\{0, 1, \ldots, n\}.$ In summary, we have  $|SL_n(F) \backslash M_n(F) \slash SL_n(F)|=n+|F|-1.$

\end{proof}

\begin{cor} \label{cor:upper-bound-sln}
    Let $\Gamma(M_n(F), S)$ be an $SL_n(F)$-unitary Cayley graph. Then $\Gamma(M_n(F), S)$ has at most $n+|F|-1$ distinct eigenvalues. 
\end{cor}

\begin{rem}
    We remark that the upper bound in \cref{cor:upper-bound-sln} is also optimal in various cases. For example, for $n=2$, $\F_p$, $R=M_2(\F_p)$ and $U=SL_2(\F_p)$ with $p \in \{5, 7\}$, the number of distinct eigenvalues in $\Gamma(R,U)$ is exactly $p+1=n+|\F_p|-1.$ However, this upper bound is not optimal for $p=2, 3.$
\end{rem}
Let us now study the $U$-unitary Cayley graphs over the matrix ring $M_n(R)$ where $R$ is a local ring. In general, the theory of matrix rings over $R$ is quite complicated. However, when $R$ is a principal local ideal ring (PIR), there is a quite elegant theory that classifies equivalent classes of matrices over $M_n(R).$ First, we remark that a finite local PIR is necessarily a symmetric Frobenius ring (see \cite{honold2001characterization}).  By \cref{prop:matrix-frobenius}, $M_n(R)$ is also a symmetric Frobenius ring.  As a result, \cref{prop:eigenvalues} applies to all $U$-unitary Cayley graphs over $M_n(R)$. To describe the double coset $GL_n(R) \backslash M_n(R) \slash M_n(R)$, we first discuss some terminology and notation. Let $\pi$ be the uniformizer for the maximal ideal of $R.$ Then, there exists a positive integer $m$ such that $\pi^{m-1} \neq 0$ but $\pi^{m}=0.$ Furthermore, every elements in $R$ can be written in the form $u\pi^a$ where $u \in R^{\times}$ and $0 \leq a \leq m.$

\begin{prop}(\cite[Theorem 15.24]{brown1993matrices}])
Let R be a local PIR with $\pi$ as a uniformizer. Then every $A \in M_n(R)$ is $GL_n(R)$-equivalent to a unique diagonal matrix $D$ of the form 
\[ D ={\rm diag}(p^{a_1}, \ldots , p^{a_n}) ,\]
where $0 \leq a_1 \leq a_2\leq  \cdots \leq a_n \leq m$ ($D$ is called \textit{the} Smith normal form of $A$). Consequently, the number of classes in $GL_n(R) \backslash M_n(R) \slash GL_n(R)$ is precisely ${m+n \choose n}.$
    
\end{prop}

\begin{cor} \label{cor:number-eigs-local-rings}
Let R be a local PIR and $\Gamma(M_n(R),S)$ be a gcd-graph over $M_n(R).$ Then $\Gamma(R,S)$ has at most ${m+n \choose n}$ distinct eigenvalues. 
\end{cor}

\begin{rem}
    The upper bound in \cref{cor:number-eigs-local-rings} could be optimal. For example, let $n=2$ and $R=\Z/4$ (so $m=2$ in this case). Let $S$ be set of all $A \in M_n(R)$ such that or $A$ is $GL_n(R)$-equivalent to one of the following matrices 
    \[ X_1 = \begin{bmatrix}
        1 & 0 \\  0 & 0 
    \end{bmatrix}, X_2 =\begin{bmatrix}
        2 & 0 \\  0 & 0 
    \end{bmatrix}.  \]
Then, the gcd-graph $\Gamma(R,S)$ has the following characteristic polynomial 
\[ (x - 81)  (x + 15)^6 (x - 17)^9 (x - 9)^{72}  (x + 7)^{72}  (x + 3)^{96} . \]
It has exactly $6 = {m+n \choose m} = {4 \choose 2} $ distinct eigenvalues. 
    
\end{rem}

\section{Perfect state transfer on $U$-unitary Cayley graphs}
Let $G$ be an undirected simple graph with adjacency matrix $A_G$. Let $F(t)$ be the continuous-time quantum walk associated with $G$; namely, $F(t) = \exp(\mathrm{i}A_G t)$. We say that there is perfect state transfer (PST) in graph $G$ if there are distinct vertices $a$ and $b$ and a positive real number $t$ such that $|F(t)_{ab}| = 1$. The concept of PST was introduced in \cite{christandl2004perfect} in the context of quantum spin networks. Since this pioneering work, there has been a series of articles studying this phenomenon on arithmetic graphs (see \cite{bavsic2009perfect, cheung2011perfect, godsil2012state, saxena2007parameters} for some works in this line of research). One of the main reasons to focus on arithmetic graphs is that a regular graph that has PST must be necessarily integral; namely, all of its eigenvalues are integers (see \cite{godsil2012state, saxena2007parameters}). So shows in \cite{so2006integral} that a circulant graph is integral if and only if it is a gcd-graph. Building upon the work of \cite{bavsic2009perfect, bavsic2013characterization, saxena2007parameters}, we show in \cite[Theorem 3.5]{nguyen_pst} that if PST exists, it will only happen at vertices with some strict local conditions. More precisely, in \cite{nguyen_pst} for a finite commutative Frobenius ring, we describe the necessary and sufficient conditions for a $U$-unitary Cayley graph to have PST (see \cite[Theorem 2.2]{nguyen_pst}).

In this section, we generalize these results to the non-commutative setting. In particular, we will show that there is no PST on gcd Cayley graphs over $M_n(F)$ where $n > 1$ and $F$ is a finite field.

First, we show that \cite[Theorem 2.2]{nguyen_pst} generalizes without much modification to the non-commutative setting (the underlying reason is that our argument mostly uses the additive structure of $R$). Let $R$ be a symmetric Frobenius $\mathbb{Z}/n$-algebra. Let $\psi: R \to \mathbb{Z}/n$ be the non-degenerate functional of $R$, and let $\chi$ be the associated generating character of $\Hom(R, \mathbb{C}^{\times})$. Let $G = \Gamma(R,S)$ be a \textit{generic} Cayley graph over $R$. The adjacency matrix $A_G$ is an $R$-circulant matrix (see \cite{kanemitsu2013matrices, chebolu2022joins}). For each $r \in R$, let 
\[ \vec{v}_r = \frac{1}{\sqrt{|R|}}[\chi_r(s)]_{s \in R}^{T} \in \mathbb{C}^{|R|} .\] 
By the circulant diagonalization theorem (see \cite{kanemitsu2013matrices}), the set $\{\vec{v}_r : r \in R\}$ forms a normalized orthonormal eigenbasis for all $A_G$. Furthermore,  the eigenvalues of $A_G$ are precisely the multiset $\{\lambda_r \}_{r \in R}$ where $\lambda_r = \sum_{s \in S} \chi_r(s).$ 
Let $V = [\vec{v}_r]_{r \in R} \in \C^{|R| \times |R|}$ be the matrix formed by this eigenbasis and $V^{*}$ be the conjugate transpose of $V$. Then we can write
\[ A_G = V D V^*=\sum_{r \in R} \lambda_r \vec{v}_r\vec{v}_r^{*},\]
here $D =\text{diag}([\lambda_r]_{r \in R})$ is the diagonal matrix formed by the eigenvalues $\lambda_r.$ We then have
\[ F(t)= \sum_{r \in R} e^{\mathrm{i} \lambda_r t} \vec{v}_r \vec{v}_r^{*}. \]
Hence
\[ F(t)_{s_1, s_2} =  \frac{1}{|R|} \sum_{r \in R} e^{\mathrm{i} \lambda_r t} \chi_r(s_1-s_2)= \frac{1}{|R|} \sum_{r \in R} e^{\mathrm{i} \lambda_r t} \zeta_n^{\psi(r(s_1-s_2))} = \frac{1}{|R|} \sum_{r \in R} e^{2 \pi \mathrm{i} \left(\lambda_r \frac{t}{2 \pi}+\frac{ \psi(r(s_1-s_2))}{n}\right)}. \]
By the triangle inequality, $|F(t)_{s_1, s_2}|=1$ if and only if $\lambda_r \frac{t}{2 \pi} + \frac{\psi(r(s_1-s_2))}{n}$ are constant modulo $1.$ Furthermore, by symmetry, there exists perfect state transfer between $s_1$ and $s_2$ if and only if there exists perfect state transfer between $0$ and $s_2-s_1.$ We then have the following criterion, which is a direct generalization of \cite[Theorem 4]{bavsic2009perfect} and \cite[Theorem 2.2]{nguyen_pst}.
\begin{thm} \label{thm:condition-pst}
    There exists perfect state transfer from $0$ to $s$ at time $t$ if and only if for all $r_1, r_2 \in R$
    \[ (\lambda_{r_1}-\lambda_{r_2}) \frac{t}{2 \pi} + \frac{\psi((r_1-r_2)s)}{n} \equiv 0 \pmod{1}.\]
\end{thm}
We will now apply \cref{thm:condition-pst} to some classes of $U$-unitary Cayley graphs. First, we discuss some necessary conditions. Let $\Gamma(R,S)$ be a $U$-unitary Cayley graph that has PST. Let $\Delta := \Delta_{S}$ be the abelian group generated by $r_1 - r_2$, where $r_1$ and $r_2$ are elements of $R$ such that $\lambda_{r_1} = \lambda_{r_2}$. By \cref{thm:condition-pst}, we must have $\psi(ds) = 0$ for all $d \in \Delta$. In particular, if $\Delta = R$, then by the non-degeneracy of $\psi$, $s$ must be $0$, and hence PST cannot exist on $\Gamma(R,S)$. This is the case when $R = M_n(F)$.

\begin{prop} \label{prop:pst-matrix-gl-n}
    Let $R=M_n(F)$ where $n>1$ and $F$ is a finite field. Let $U=R^{\times}= GL_n(F)$. Suppose that $\Gamma(R,S)$ is a $U$-unitary Cayley graph. Then, there is no PST on $\Gamma(R,S).$
\end{prop}

\begin{proof}
    By \cref{prop:supercharacter-U}, if $u_1, u_2 \in GL_n(F)$, then $\lambda_{u_1} =\lambda_{u_2}.$ Consequently, $u_1-u_2 \in \Delta.$ Furthermore, by \cite{invertible_matrices}, every matrix in $M_n(F)$ is the difference of two invertible matrices. Therefore, $\Delta=M_n(F).$ This shows that there is no PST on $\Gamma(R,S).$
\end{proof}
\begin{rem}
    The same argument works for a ring of the form $M_n(S)$ where $S$ is a finite Frobenius local commutative ring (and more generally, for rings which are a direct product of these $M_n(S)$.) In fact, let $F$ be the residue field of $R$.
 Then, there is a natural ring homomorphism $M_n(S) \to M_n(F).$ Under this map, a matrix $A \in M_n(S)$ is invertible if and only if its image $\bar{A} \in M_n(F)$ is also invertible. As a result, every matrix in $M_n(R)$ is also a sum of two invertible matrices. 
 \end{rem}

We now show that the same statement holds if we take $U=SL_n(F).$
 
\begin{prop} \label{prop:pst-matrix-sl-n}
    Let $R=M_n(F)$ where $n>1$ and $F$ is a finite field. Let $U=SL_n(F)$. Suppose that $\Gamma(R,S)$ is a $U$-unitary Cayley graph. Then, there is no PST on $\Gamma(R,S).$
\end{prop}

\begin{proof}
    Let $V_1$ be the set of rank $1$ matrices in $M_n(F)$. By the proof of \cref{prop:similar-sln}, for $A,B \in V_1$, $A \sim_{SL_n(F)} B.$ As a result, $\lambda_{A}=\lambda_{B}.$ Let $\Delta'$ be the abelian group generated by $(A-B)$ where $A, B \in V_1.$ Then, $\Delta' \subset \Delta.$ We claim that $\Delta'=M_n(F).$ In fact, for $v \in \F^n$, we can write $v=v_1-v_2$ where $v_1, v_2 \neq 0.$ Therefore, if $X$ is a matrix with exactly one non-zero column vector, then $X \in \Delta'.$ Since every matrix is a sum of those $X$'s, we conclude that $\Delta'=M_n(F).$ Consequently, $\Delta= M_n(F)$ and hence $\Gamma(R,S)$ has no PST. 
\end{proof}

\begin{rem}
By \cref{prop:pst-matrix-gl-n} and \cref{prop:pst-matrix-sl-n}, PST cannot exist on $U$-unitary graphs with $U = GL_n(F)$ or $U=SL_n(F).$ However, 
    PST can exist if $U$ is a smaller subgroup of $GL_n(F)$. For example, let us consider the case $R=M_2(\F_2)$ and $U$ is the subgroup of permutation matrices in $R$; namely 
    \[ U  = \left\{ \begin{bmatrix}
        1 & 0 \\  0 & 1 
    \end{bmatrix}, \begin{bmatrix}
        0 & 1 \\  1 & 0 
    \end{bmatrix} \right \}  \]
    Then $\Gamma(R,U)$ is a disjoint union of four copies of $C_4$ where $C_4$ is the cycle graph on $4$ vertices. It is known that there is PST on $C_4$ (see \cite[Lemma 9]{bavsic2009perfect}.) It would be interesting to classify all pairs $(M_n(F), U)$ such that there exists a $U$-unitary graph that has PST. 
\end{rem}

\section*{Acknowledgements}
We thank  J\'an Min\'a\v{c} for his interest in this project and for asking some interesting questions about the upper bound in \cref{prop:eigenvalues}.

\end{document}